\documentclass[12pt]{amsart}
\usepackage[english]{babel}
\usepackage{amsfonts,amssymb,latexsym,amscd}

\theoremstyle{plain}
\newtheorem{theorem}{Theorem}
\newtheorem{lemma}{Lemma}
\newtheorem{proposition}{Proposition}
\newtheorem{corollary}{Corollary}
\theoremstyle{definition}
\newtheorem{definition}{Definition}
\theoremstyle{remark}

\usepackage{hyperref}

\begin{document}

\title{Equivariant Movability of Topological Groups}
\author{Pavel S. Gevorgyan}

\address{Moscow Pedagogical State University}

\email{pgev@yandex.ru}

\begin{abstract}
The equivariant movability of topological spaces with an action of a given topological group $G$ is considered. In particular, the equivariant movability of topological groups is studied. It is proved that a second countable group $G$ is Lie if and only if it is equivariantly movable.
\end{abstract}


\maketitle

\section{Introduction}

The first results on the equivariant movability of $G$-spaces
were obtained in \cite{gev2} and \cite{gev1}.
If $X$ is a $p$-paracompact space and $H$ is a closed subgroup of a topological group $G$, then the $G$-movability of $X$ implies its $H$-movability \cite[Theorem 3.3]{gev1}.
If $X$ is a metrizable $G$-movable space and $H$ is a closed
normal subgroup of the topological group $G$,
then the $H$-orbit space  $X|_H$ with the natural action
of the group $G$ is $G$-movable as well \cite[Theorem 6.1]{gev1}.
In the case of $H=G$, the equivariant movability
of a metrizable $G$-space implies the movability
of the orbit space $X|_G$ \cite[Corollary 6.2]{gev1}.
The  converse is generally false \cite[Example 6.3]{gev1}.
However, if $X$ is metrizable, $G$ is a compact Lie group, and the
action of $G$ on $X$ is free, then
the equivariant movability of the $G$-space $X$ is equivalent to the
movability of the orbit space $X|_G$ \cite[Theorem 7.2]{gev1}.
The $G$-movability of $X$ implies also the movability of the $H$-fixed point
space  $X[H]$ \cite[Theorem 4.1]{gev1}.
In particular, the equivariant movability of a $G$-space $X$
implies the movability of the topological space
$X$ \cite[Corollary 3.5]{gev1}. The converse is not true
even for the cyclic group $Z_2$.
In \cite{gev1}, an example of a  $Z_2$-space $X$
which is movable but not $Z_2$-movable was constructed
\cite[Example 5.1]{gev1}.

The movability of topological groups in classical shape theory
was studied by Keesling \cite{keesl75, keesl74} and by
Kozlovskii with Segal \cite{ks76}.
 In particular, Keesling \cite{keesl74} proved that,
for compact connected Abelian groups, movability
is equivalent to local connectedness. In this paper,
we study the equivariant movability of topological groups; in particular, we
prove that a second countable compact
topological group $G$ is a Lie group if and only if it is
equivariantly movable (Theorem \ref{th-main}).
This theorem provides new examples of spaces which are movable but not equivariantly movable.

\section{Preliminaries on Equivariant Topology and Equivariant Shapes}\label{sec1}

Let $G$ be a topological group. A topological space $X$ is called
a $G$-space if there is a continuous map $\theta :G \times X \to
X$ of the direct product $G \times X$ to $X$, for which we use the
notation $\theta(g, x) =
gx$,
such that
$$ (1) \quad g(hx)=(gh)x\qquad\text{and}\qquad (2) \quad ex=x $$
for $g, h \in G$ and $x\in X$,
where $e$ denotes the identity element of $G$.
The (continuous) map $\theta :G \times X \to X$ is called a
(continuous) action of the group $G$ on the topological space $X$.
An evident example is the so-called trivial action of $G$ on $X$ defined by
$gx=x$ for all $g\in G$ and $x\in X$. Another example is the action
of the group $G$ on itself defined by $(g, x) \to gx$ for all
$g\in G$ and $x\in G$.

Let $H$ be a closed subgroup of a group $G$.
There exists a natural action of the group $G$ on space
$G|H$, which is defined by $g(g'H)=(gg')H$.


A subset $A$ of a $G$-space $X$ is said to be invariant if  $ga\in A$ for
any $g\in
G$ and $a\in A$. Obviously,  an invariant
subset of a $G$-space is itself a $G$-space. If $A$ is an
invariant subset of a $G$-space $X$, then every neighborhood of
$A$ contains an open invariant neighborhood of $A$ (see
\cite[Proposition 1.1.14]{p60}).

Let $X$ and  $Y$ be $G$-spaces. A (continuous) map $f: X\to Y$ is
called a $G$-map, or an equivariant map, if $f(gx)=gf(x)$ for
any $g\in G$ and $x\in X$. Note that the identity map $i:X \to X$
is equivariant, and the composition of any equivariant maps is
equivariant. Therefore, all $G$-spaces and equivariant maps form a
category, which we denote by $\mathtt{Top}^G$.

Let $Z$ be a $G$-space, and let $Y\subseteq Z$ be an invariant
subset. A $G$-retraction of $Z$ to $Y$ is a $G$-map $r:Z \to Y$
such that $r|_Y=1_Y$.

Let $K_G$ be a class of $G$-spaces. A $G$-space $Y$ is called a
$G$-absolute neighborhood retract for the class $K_G$,  or a
$G$-ANR$(K_G)$ (a $G$-absolute retract for the class $K_G$,  or a
$G$-AR$(K_G)$), if $Y\in K_G$ and, whenever $Y$ is a closed
invariant subset of a $G$-space $Z\in K_G$, there exists an
invariant neighborhood $U$ of $Y$ and a $G$-retraction $r:U \to Y$
(there exists a $G$-retraction $r:Z \to Y$).

Let $X$ and $Y$ be $G$-spaces. We say that two equivariant maps, or $G$-maps,
$f_0, f_1 :X\to Y$ are $G$-homotopic,
or equivariantly homotopic, and write $f_0\simeq_G f_1$
if there exists a $G$-map $F:X\times I \to Y$ such that $F(x,0)=f_0(x)$
and $F(x,1)=f_1(x)$ for all $x\in X$ (we assume that $G$ acts
trivially on $I$). The relation $\simeq_G$ is an equivalence relation;
we denote the $G$-homotopy class
of a $G$-map $f$ by $[f]$. In this way, we obtain the category
$H$-$\mathtt{Top}^G$,
whose objects are $G$-spaces and whose morphisms are classes
of $G$-homotopic $G$-maps.

\begin{definition}
An inverse $G$-ANR system $\underline{X}=\{ X_\alpha , p_{\alpha
\alpha '}, A \}$ in $\mathtt{Top}^G$ (where all $X_\alpha$ are $G$-ANRs)
is said to be associated with a $G$-space $X$ if there exist $G$-maps
$p_\alpha:X\to X_\alpha$ such that
$p_\alpha\simeq_G p_{\alpha\alpha'}\circ p_{\alpha'}$ and the following
two conditions hold for every $G$-ANR $P$:

(1) for every $G$-map $f:X\to P$, there is an $\alpha\in A$
and a $G$-map $h_\alpha:X_\alpha \to P$ such that
$h_\alpha\circ p_\alpha\simeq_G f$
(i.e., each $f$ factors through some $X_\alpha$);

(2) if $\varphi\circ p_\alpha \simeq_G \psi\circ p_\alpha$ for some
$G$-maps $\varphi, \psi :X_\alpha \to P$,
then there is an $\alpha'\geqslant \alpha$ such
that $\varphi \circ p_{\alpha\alpha'}\simeq_G \psi \circ p_{\alpha\alpha'}$.
\end{definition}

For every $G$-space $X$, there exists an inverse $G$-ANR
system $\underline{X}=\{ X_\alpha , p_{\alpha \alpha '}, A \}$ associated
with $X$; i.e., the category $H$-ANR$^G$ is dense in $H$-$\mathtt{Top}^G$ \cite{am87}.

\begin{definition}
An inverse $G$-ANR system $\underline{X}=\{ X_\alpha , p_{\alpha
\alpha '}, A \}$ is said to be equivariantly movable, or $G$-movable,
if, for every $\alpha \in A$,
there exists an $\alpha ' \in A$ such that $\alpha '\geqslant \alpha $ and,
for any  $\alpha '' \geqslant \alpha $ ($\alpha '' \in A$),
there exists a $G$-homotopy class $r^{\alpha ' \alpha ''} :
X_{\alpha '} \to X_{\alpha ''}$ for which
$$p_{\alpha \alpha ''} \circ r^{\alpha ' \alpha ''} = p_{\alpha
\alpha '} .$$
\end{definition}

\begin{definition}
A $G$-space $X$ is said to be equivariantly movable, or $G$-movable, if
there exists an equivariantly movable inverse $G$-ANR system
$\underline{X}=\{ X_\alpha , p_{\alpha \alpha '}, A \}$ associated
with $X$.
\end{definition}

Note that the last definition of equivariant movability coincides
with that of ordinary movability for the trivial group $G=\{e\}$.

The reader is referred to the books by K. Borsuk \cite{b72} and by
S. Marde\v{s}i\'{c} and J. Segal \cite{ms82} for general
information about shape theory and to the book by G. Bredon
\cite{br72} for an introduction to compact transfor\-mation groups.

\section{Weakly Equivariantly Shape Comparable $G$-Spaces}

\begin{definition}\label{def-1}
We say that $G$-spaces $X$ and $Y$ are weakly equivariantly shape comparable if there exist $G$-shape morphisms both from
$X$ to $Y$ and from $Y$ to $X$.
\end{definition}

Obviously, this relation is an equivalence.
Therefore, the family of all $G$-spaces splits
into disjoint classes of weakly equivariantly shape comparable
$G$-spaces.
We denote the class of spaces weakly equivariantly
shape comparable with a $G$-space $X$ by $wes(X)$.

Let $wes(*)$ be the weak equivariant shape
comparability class of the one-point $G$-space $\{*\}$.
The following proposition characterizes the class $wes(*)$.

\begin{proposition}
The family of all $G$-spaces with a fixed point is precisely
the weak equivariant shape comparability class $wes(*)$.
\end{proposition}

\begin{proof}
It is sufficient to prove that any
$G$-space $X\in wes(*)$ has a fixed point. Indeed, let
$F:\left\{ \ast \right\}\to X$ be a $G$-shape morphism. We
regard the $G$-space
$X$ as an invariant closed subset of some
$G$-AR  $Y$. Since $F:\left\{ \ast \right\}\to
X$ is a $G$-shape morphism, it follows that any invariant
neighborhood of the space
$X$ in $Y$ has a fixed point. Therefore, the $G$-space $X$ has a fixed
point as well, because the set of all fixed points of the $G$-space $Y$
is closed.
\end{proof}

We denote the weak equivariant shape comparability class of
a group $G$ with its natural action on itself
by $wes(G)$.
The following theorem characterizes the class $wes(G)$
in the case of a second countable compact group.

\begin{theorem}\label{th-1}
Let $G$ be a second countable compact group. Then the $G$-space
$X$ belongs to the class $wes(G)$ if and only if $X=X|G\times G$.
\end{theorem}

The proof of this theorem is based on the following theorem of
independent interest.

\begin{theorem}\label{th-2}
Suppose that $G$ is a second countable compact group, $H\subset G$
is a closed normal
subgroup of $G$, and $X$ is any $G$-space.
If there exists a $G$-shape morphism $F:X\to G|H$, then $X=G\times_HA$
and $F$ is generated by the $G$-equivariant
map $h:G\times_HA\to G|H$  given by $h([g,a])=gH$, where
$A$ is some $H$-space, $G\times_HA$ is the twisted product,
and $[g,a]$ is the $H$-orbit of the point $(g,a)$.
\end{theorem}

\begin{proof}
Since the group $G$ is compact and
second countable, it follows that so is the group $G\left| H \right.$.
By a well-known theorem of Pontryagin \cite[Theorem 68]{pontr},
there exist closed normal divisors $K_i$ ($i=1,2,\dots$)
of the group $G\left| H \right.$
such that $K_{i+1} \subset K_i $,  $(G\left| {H)\left| {K_i } \right.}
\right.$ is a Lie group for any $i=1, 2, \dots$, and
\begin{equation}
G|H= \lim_{\longleftarrow}\left\{(G|H)|K_i; \ p_{i, i+1}\right\},
\end{equation}
where the $p_{i,i+1} :(G|H)|K_{i+1} \to (G|H)|K_i$ are the natural
epimorphisms generated by the inclusions $K_{i+1} \subset K_i $.

Note that, for any $i=1, 2, \dots$, the group $(G|H)|K_i$
is isomorphic (topologically and algebraically) to the group $G|\tilde {K}_i$,
where $\tilde {K}_i =p^{-1}\left( {K_i } \right)$ and
$p:G\to G\left| H \right.$ is the natural epimorphism.
The groups $\tilde {K}_i$ with $i=1, 2, \dots$, being
continuous preimages of the closed
normal subgroups $K_i \subset G\left| H \right.$, are closed and normal.
Thus, we have
\begin{equation}
G | H=\lim_{\longleftarrow}\left\{ G | \tilde {K}_i ;p_{i,i+1}\right\},
\end{equation}
where all maps $p_{i,i+1}$ are $G$-equivariant, provided that the space
$G | {\tilde {K}_i }$ is endowed with the natural action
of the group $G$.

Now, suppose given a $G$-shape morphism $F:X\to G|H$. Let us prove that $F$
is generated by some equivariant map $h:X\to G|H$.

Consider the $G$-shape maps
$S_G \left( {p_i } \right)\circ F:X\to G|\tilde {K}_i$,
where $S_G$ is the $G$-shape functor and
the $p_i:G\to G|\tilde{K}_i$ are the projections. Since
all $G|\tilde{K}_i$ are $G$-ANRs, it follows that
there exist $G$-equivariant maps $f_i :X\to G|\tilde{K}_i$ for which
\begin{equation}
S_G \left( {p_i } \right)\circ F=S_G \left( {f_i } \right).
\end{equation}
On the other hand, we have
\begin{multline*}
S_G \left( {f_i } \right)=S_G \left( {p_i } \right)\circ F=S_G \left(
{p_{i,i+1} } \right)\circ S_G \left( {p_{i+1} } \right)\circ F \\
=S_G \left({p_{i,i+1} } \right)\circ S_G \left( {f_{i+1} } \right)=S_G \left(
{p_{i,i+1} \circ f_{i+1} } \right).
\end{multline*}
Thus, the $G$-maps $f_i $ and $p_{i,i+1} \circ f_{i+1} $ to the
$G$-ANR $G|\tilde {K}_i$ generate the same $G$-shape morphism.
Therefore, they are $G$-homotopic, i.e.,
\begin{equation}
f_i \simeq_G p_{i,i+1} \circ f_{i+1} .
\end{equation}
Let $h_1 =f_1 $. All maps $p_{i,i+1}$ are $G$-fibrations.
Hence there exists a $G$-map $h_2 :X\to G|\tilde{K}_2$
such that $h_2 \simeq_G f_2$ and $h_1 =p_{1,2} \circ h_2$.
Continuing this construction by induction,
we obtain $G$-maps $h_i :X\to G|\tilde {K}_i$
such that
\begin{equation}
h_i \simeq_G f_i\qquad\text{and} \qquad h_i =p_{i,i+1} \circ h_{i+1}.
\end{equation}

We set $h=\lim\limits_{\longleftarrow}h_i $.
Obviously, $h$ is a $G$-map from $X$ to $G\left| H
\right.$. Let us prove that the map $h$ has the required property
\begin{equation}
\label{eq5}
S_G \left( h \right)=F.
\end{equation}
Indeed, the continuity of the shape functor \cite{holsz} implies
\[
S_G \left( h \right)=S_G \left( {\,h_i } \right)=\,S_G \left( {h_i }
\right)=\,S_G \left( {f_i } \right)=\,S_G \left( {p_i } \right)\circ
F=1\circ F=F.
\]

Let $A=h^{-1}\left( {eH} \right)$. Then $A$ is an $H$-invariant
subset of the $G$-space $X$,
$X=G\times_HA$, and $h:G\times_HA\to G|H$ is defined by
$h([g,a])=gH$ \cite[Proposition 3.2]{br72}.

This completes the proof of the theorem.
\end{proof}

\begin{proof}[Proof of Theorem \ref{th-1}] First, note that if
$X=A\times G$, where $A$ is a trivial $G$-space,
then the  map $f:A\times G\to G$
defined by $f(a,g)=g$ is equivariant. Therefore, $X=A\times G\in wes(G)$.

Now, suppose that $X$ is a $G$-space in the class $wes(G)$.
Consider a $G$-shape morphism $F:X\to G$.
Taking the trivial group $\{e\}$ for the closed normal
subgroup $H$  in Theorem \ref{th-2}, we obtain $X=X|G\times G$.
\end{proof}

\begin{corollary}
Let $G$ be a second countable compact group,
and let  $H$ and $K$ be its closed normal subgroups.
Then  $wes(G|H)=wes(G|K)$ if and only if the subgroup $H$
is conjugate to the subgroup $K$.
\end{corollary}

\begin{proof}
By Definition \ref{def-1}, the equality $wes(G|H)=wes(G|K)$
means the existence of $G$-shape morphisms both from $G|H$ to $G|K$
and from $G|K$ to $G|H$. According to Theorem \ref{th-2},
there exist equivariant maps
from $G|H$ to $G|K$ and from $G|K$ to $G|H$,
which is possible
if and only if $H$ is conjugate to $K$ \cite[Corollary 4.4]{br72}.
\end{proof}

\section{Equivariantly Movable Groups}

\begin{definition}\label{def-2}
We say that an inverse $G$-ANR
sequence $\left\{X_K, p_{k,l}\right\}$ is
canonical if, for any $k\in N$, there exists an
equivariant map $r^{k,k+1}:X_k \to X_{k+1}$ such that
\begin{equation}\label{eq1}
p_{j,k+1} \circ r^{k,k+1}\simeq_Gp_{j,k},
\end{equation}
where $1\leqslant j<k$.
\end{definition}

The following proposition is easy to prove.

\begin{proposition}\label{th-3}
Any canonical $G$-ANR sequence is $G$-movable.
\end{proposition}

The following assertion is also valid.

\begin{proposition}\label{th-4}
Any $G$-movable inverse $G$-ANR sequence contains a canonical
subsequence.
\end{proposition}

\begin{proof}
Let $\left\{X_k, p_{k,l}\right\}$ be a $G$-movable $G$-ANR sequence.
For $k=1$,
there exists a number $k_1 >1$ which witnesses the
sequence $\left\{X_k, p_{k,l}\right\}$
being $G$-movable. By induction, given a number $k=k_i$, we choose
a number $k_{i+1} > k_i$ such that, for any other number
$l>k_{i+1}$, there exists
an equivariant map $r^{k_{i+1},l} : X_{k_{i+1}} \to X_l$ for which
\begin{equation}\label{eq3}
p_{k_i,l} \circ r^{k_{i+1},l} \simeq_G p_{k_i,k_{i+1}}.
\end{equation}
It is easy to check that
$\left\{X_{k_i}, p_{k_i,k_{i+1}}, i\geqslant 1\right\}$
is the required canonical subsequence.
\end{proof}

The following theorem follows directly from Propositions~\ref{th-3}
and \ref{th-4}.

\begin{theorem}\label{th-5}
A compact metrizable $G$-space $X$ is $G$-movable if and only
if there exists
a canonical inverse $G$-ANR sequence
$G$-associated with $X$.
\end{theorem}

\begin{lemma}\label{lemma-2}
Let $X$ be a compact metrizable $G$-movable space. Then
there exists an
inverse $G$-ANR sequence
$\left\{X_k, p_{k,l}\right\}$  $G$-associated with $X$
such that $X_k \in wes(X)$ for any
$k\in N$.
\end{lemma}

\begin{proof}
Let $X$ be a compact metrizable $G$-movable space.
By Theorem~\ref{th-5}, there exists a  canonical
$G$-ANR sequence
$\left\{X_k, p_{k, l}\right\}$ $G$-associated with $X$.
Let us prove that this sequence is as required, i.e.,
$X_k \in wes(X)$ for any $k\in N$.
Since the maps $p_k:X\to X_k$ are equivariant, it suffices to prove
the existence of a $G$-shape morphism from $X_k$ to $X$.
Consider the morphism
$F=\left\{f_i, \varphi\right\} :
X_k \to \left\{X, p_{i,i'}, i\geqslant k\right\}$ defined by
\begin{equation}
\varphi(i)=k, \quad f_i =p_{i,i+1} \circ r^{k,i+1},
\end{equation}
where $r^{k,i+1}=r^{i,i+1}\circ \dots \circ r^{k,k+1}$
and the $r^{j,j+1}$ are the equivariant maps mentioned in
Definition~\ref{def-2}. The morphism
$F=\left\{f_i, \varphi\right\}$ is a map
of $G$-ANR sequences. Indeed,
\begin{multline*}
p_{i,i+1} \circ f_{i+1} =p_{i,i+1} \circ p_{i+1,i+2} \circ
r^{k,i+2}=p_{i,i+2} \circ r^{i+1,i+2}\circ r^{k,i+1} \\
\simeq_G p_{i,i+1} \circ r^{k,i+1}=f_i.
\end{multline*}
\end{proof}

Lemma~\ref{lemma-2} implies directly the following assertion.

\begin{theorem}
Any compact metrizable $G$-movable space is weakly equivariantly shape
comparable with some
$G$-ANR.
\end{theorem}

However, as shown below (see Corollary~\ref{cor-1}),
there exist compact metrizable $G$-spaces
which are not weakly equivariantly shape comparable with any
$G$-ANR ($G$-movable space).

\begin{lemma}\label{lemma-1}
Suppose that $ass_G X=\left\{X_k, p_{k,k+1}\right\}$,
$ass_G Y=\left\{Y_l, q_{l,l+1}\right\}$,
$wes(X)=wes(Y)$, and $X_k \in K$
for any $k\in N$, where $K$ is a class
of weak equivariant shape comparability.
Then the $G$-ANR sequence $\left\{Y_l, q_{l,l+1}\right\}$
has a subsequence $\left\{Y_{l_i}, q_{l_i, l_{i+1}}\right\}$
such that $Y_{l_i} \in K$ for any
$i=1,2,\dots$\,.
\end{lemma}

\begin{proof}
The condition $wes(X)=wes(Y)$ means the existence of
$G$-shape morphisms $F:Y\to X$ and $\Phi:X\to Y$.
Suppose that $F=\left\{f_k, \varphi\right\} : \left\{Y_l, q_{l,l+1}\right\}
\to \left\{X_k, p_{k,k+1}\right\}$ and
$\Phi=\left\{g_l, \psi\right\} : \left\{X_k, p_{k,k+1}\right\}\to
\left\{Y_l, q_{l,l+1}\right\}$.
Let us prove that $\left\{Y_{\varphi(k)}, q_{\varphi(k), \varphi(k+1)}\right\}$
is the required subsequence. For this purpose, we show that
the $Y_{\varphi(k)}$ and the $X_k$ are weakly equivariantly shape comparable
and belong to the class $K$ for any $k=1,2,\dots$\,.
Indeed, the maps $f_k : Y_{\varphi(k)} \to X_k$ are equivariant.
On the other hand, the $X_k$ and the $X_{\psi(\varphi(k))}$
belong to the class $K$, and
the maps $g_{\varphi(k)} : X_{\psi(\varphi(k))} \to Y_{\varphi(k)}$ are
equivariant. Therefore, $Y_{\varphi(k)} \in K$, as required.
\end{proof}

\begin{theorem}\label{th-6}
If a  weak equivariant shape comparability class $K$
contains a $G$-movable metrizable compact space, then,
for any compact metrizable $G$-space $X\in K$, there exists a
$G$-ANR sequence $\left\{X_k, p_{k,k+1}\right\}$
$G$-associated with $X$ such that $X_k \in K$ for any $k\in N$.
\end{theorem}

\begin{proof}
Let $Y\in K$ be a $G$-movable metrizable compact space.
According to Lemma~\ref{lemma-2}, there exists an
inverse $G$-ANR sequence $\left\{Y_l, q_{l,l+1}\right\}$
$G$-associated with $Y$ such that $Y_l \in K$ for any
$k\in N$. Consider any compact metrizable $G$-space $X\in K$.
By Lemma~\ref{lemma-1}, any
$G$-ANR sequence $G$-associated with $X$ has a
subsequence of spaces belonging to the class $K$.
This completes the proof of the theorem.
\end{proof}

\begin{theorem}\label{th-7}
A second countable compact group $G$ is Lie if and only if
the class $wes(G)$ contains
a second countable  $G$-movable compact space.
\end{theorem}

\begin{proof}
Necessity is obvious, because
any compact  Lie group is a $G$-ANR \cite{p60} and,
therefore, a $G$-movable space.

Let us prove sufficiency. Suppose that the class $wes(G)$
contains a second countable $G$-movable compact space.
By a theorem of Pontryagin \cite[p.~332]{pontr},
the identity element $e\in G$ is surrounded by decreasing closed normal
subgroups $K_1 \supset K_2 \supset \cdot \cdot \cdot $ such
that $G\left| {K_i }
\right.$ is a Lie group for any $i=1,2,\dots$
and $G=\lim\limits_{\longleftarrow}\left\{ {G\left| {K_i
,p_{i,i+1}} \right.} \right\}$, where the $p_{i,i+1} : G|K_{i+1} \to G|K_i$
are the natural epimorphisms
generated by the embeddings $K_{i+1} \subset K_i $.
The  group $G$ acts naturally on each $G\left| {K_i }\right.$;
all maps $p_{i,i+1}$ are $G$-equivariant with respect to these
actions, and the $G\left|
{K_i } \right.$  themselves are $G$-ANRs \cite{p60}. Thus,
the inverse sequence $\left\{G|K_i, p_{i,i+1}\right\}$ is
$G$-associated with the group $G$. By virtue of
Theorem~\ref{th-6}, the sequence $\left\{ {G\left| {K_i ,p_{i,i+1}} \right.} \right\}$
has a subsequence in which all spaces
are weakly equivariantly shape comparable with the group $G$.
This is possible only if $G\left| {K_i =G} \right.$ starting with some $i$.
Since $G\left|{K_i } \right.$ is a Lie group, the required assertion
follows.
\end{proof}

The last theorem implies directly the following criterion
for a second countable compact group to be a Lie group.

\begin{theorem}\label{th-main}
A second countable compact group is a Lie group
if and only if it is $G$-movable.
\end{theorem}

Theorem~\ref{th-main} gives new examples of movable but not $G$-movable
spaces. Indeed, as was shown by Keesling \cite{keesl75},
there exist compact connected Abelian groups which are movable
but not uniformly movable and, therefore, not Lie.
Later, Kozlovskii and Segal \cite{keesl74} constructed examples
of such groups. These groups are not $G$-movable by Theorem~\ref{th-main}.

\begin{corollary}\label{cor-1}
There exists a $G$-space which is not weakly equivariantly shape
comparable with any second countable compact $G$-movable space.
\end{corollary}

\begin{proof}
By virtue of Theorem~\ref{th-7}, any second countable compact
group not being a Lie group has the required property.
\end{proof}

\end{document}